\documentclass[12pt]{amsart}
\usepackage{latexsym,amscd,amssymb,epsfig} 
\usepackage[dvips]{color}
\usepackage[all]{xy}
\usepackage[T1]{CJKutf8}

\theoremstyle{plain}
  \newtheorem{theorem}{Theorem}[section]
  \newtheorem{proposition}[theorem]{Proposition}
  \newtheorem{lemma}[theorem]{Lemma}
  \newtheorem{corollary}[theorem]{Corollary}
  
\theoremstyle{definition}
  \newtheorem{definition}[theorem]{Definition}
  \newtheorem{example}[theorem]{Example}

\theoremstyle{remark}
  \newtheorem{remark}[theorem]{Remark}

\numberwithin{equation}{section}

\def\umapright#1{\smash{
   \mathop{\longrightarrow}\limits^{#1}}}

\def\rmapdown#1{\Big\downarrow\rlap
   {$\vcenter{\hbox{$\scriptstyle#1$}}$}}

\def\tempbaselines
{\baselineskip22pt\lineskip3pt
   \lineskiplimit3pt}
\def\diagram#1{\null\,\vcenter{\tempbaselines
\mathsurround=0pt
    \ialign{\hfil$##$\hfil&&\quad\hfil$##$\hfil\crcr
      \mathstrut\crcr\noalign{\kern-\baselineskip}
  #1\crcr\mathstrut\crcr\noalign{\kern-\baselineskip}}}\,}

\def\pullback#1&#2&#3&#4&#5&#6&#7&#8&{
\diagram{#1&\umapright{#2}&#3\cr
\rmapdown{#4}&&\rmapdown{#5}\cr
#6&\umapright{#7}&#8\cr}}

\def\calC{{\mathcal C}}
\def\calD{{\mathcal D}}
\def\calE{{\mathcal E}}
\def\calK{{\mathcal K}}
\def\calJ{{\mathcal J}}
\def\calO{{\mathcal O}}
\def\calP{{\mathcal P}}

\def\calR{{\mathcal R}}
\def\calS{{\mathcal S}}
\def\calT{{\mathcal T}}

\def\frakF{{\mathfrak F}}
\def\frakG{{\mathfrak G}}

\def\Aut{\mathop{\rm Aut}\nolimits} 
\def\coh{{\mathop{\rm coh}\nolimits}}
\def\colim{\mathop{\varprojlim}\nolimits}

\def\Hom{\mathop{\rm Hom}\nolimits} 
\def\Nat{\mathop{\rm Nat}\nolimits} 

\def\lim{\mathop{\varinjlim}\nolimits}
\def\Ob{\mathop{\rm Ob}\nolimits} 
\def\Mor{\mathop{\rm Mor}\nolimits}

\def\Res{\mathop{\rm Res}\nolimits}
\def\PSh{\mathop{\rm PSh}\nolimits}
\def\Sh{\mathop{\rm Sh}\nolimits}

\def\PG{{\calP\rtimes G}}

\def\GG{\mathbf{G}}
\def\ZZ{{\Bbb Z}}
\def\Ab{{\rm Ab}}
\def\Set{{\rm Set}}
\DeclareMathOperator{\rSet}{{\rm Set}-}

\DeclareMathOperator{\OG}{\mathbf{O}\it{G}}
\DeclareMathOperator{\pG}{\mathbf{P}\it{G}}
\DeclareMathOperator{\TG}{\mathbf{T}\it{G}}
\DeclareMathOperator{\C}{\mathbf{C}}
\DeclareMathOperator{\D}{\mathbf{D}}

\DeclareMathOperator{\rMod}{Mod-}
\DeclareMathOperator{\rmod}{mod-}

\begin{document}

\title[Sheaves in finite group representations]
{On sheaves in finite group representations}

\author{Tengfei Xiong}
\author{Fei Xu}
\email{18tfxiong@stu.edu.cn}
\email{fxu@stu.edu.cn}
\address{Department of Mathematics\\
Shantou University\\
Shantou, Guangdong 515063, China}

\subjclass[2020]{20C05, 18F10, 18A25, 18F20}

\keywords{Group representation, sheaf, $G$-set, orbit category, transporter category, (co)continuous functor, morphism of topoi, modules on sites}

\thanks{The authors \begin{CJK*}{UTF8}{}
\CJKtilde \CJKfamily{gbsn}(熊腾飞、徐 斐)
\end{CJK*} are supported by the NSFC grants No. 12171297 and No. 11671245.}

\begin{abstract}
Given a general finite group $G$, we consider several categories built on it, their Grothendieck topologies and resulting sheaf categories. For a certain class of transporter categories and their quotients, equipped with atomic topology, we explicitly compute their sheaf categories via sheafification. This enables us to identify $G$-representations with various fixed-point sheaves. As a consequence, it provides an intrinsic new proof to the equivalence of M. Artin between the category of sheaves on the orbit category and that of group representations.
\end{abstract}

\maketitle

Let $G$ be a finite group. The category $\rSet{G}$ of (right) $G$-sets and $G$-equivariant maps has been a major subject whenever group actions occur. The orbit category $\calO(G)$ is a full subcategory, consisting of the objects $H\backslash G$, for all subgroups $H\subset G$. In group representations, both categories are very useful in that they can be used to introduce Mackey functors \cite{We1}, to  study Alperin's weight conjecture \cite{Li1, We3}, and to investigate subgroup complexes and homology decompositions \cite{AKO, Sm}. 

In the above cases, functors over various categories are of fundamental importance. Following a classical result (recalled here as Theorem 1.6), of M. Artin \cite[0.6 \& 0.6 bis]{Ar} and \cite[II.1.9]{Mi}, we shall consider several Grothendieck topologies on the categories $\rSet{G}, \calO(G)$, and $G$, as well as $\calT(G)$ (\textit{the transporter category}). Artin's result asserts that a certain abelian sheaf category on $\rSet{G}$ is equivalent to the module category of $\ZZ G$. By Mac Lane and Moerdijk \cite[III.9]{MM}, $\rSet{G}$ may be replaced by the finite category $\calO(G)$. Thus representations of $G$ are identified with sheaves on $\calO(G)$. Furthermore the equivalence remains correct if the coefficient ring is changed from $\ZZ$ to $R$, an arbitrary commutative ring with identity. Under the circumstance, an $R$-representation of $G$, identified with a sheaf of $R$-modules on $\calO(G)$, glues up various $R(N_G(H)/H)$-modules. The phenomenon can be thought as a way of decomposing group representations, analogous to the homology decomposition of the classifying space $BG$ via $\calO(G)$ \cite{AKO, DH, Sm}. This asks for a better understanding of Artin's result, in the context of finite group representations. Indeed, the recent work of Balmer \cite{Ba} indicates that sheaf theoretic methods may find more applications in this area, apart from the well-known Deligne-Lusztig theory. As an attempt, we try to extract more information from Theorem 1.6. 

Take a finite $G$-poset $\calP$ and consider the abstract transporter category $\PG$. Whenever $\calP$ has an initial object, it is possible to introduce the atomic topology on $\PG$ and on an arbitrary quotient category $\calC$. We are able to compute the sheaf category over $\calC$, by sheafifying presheaves. Two natural (continuous and/or cocontinuous) functors between $G, \PG$ and $\calC$ will induce equivalences among their sheaf categories. Based on these, we obtain far more general statements, see Theorems 3.8 and 4.4, than Artin's original one, where $\calP$ is taken to be the poset of all subgroups of $G$. It allows further applications of sheaf-theoretical methods to finite group representations.

The paper is organized as follows. In Section 1, we recall the definitions of a site and of a sheaf on a site, and the result of Artin (especially a version of Mac Lane-Moerdijk), accompanied by motivating examples. Then in Section 2 we state the concepts of continuous and cocontinuous functors, for comparing sheaf categories. Through a certain class of abstract transporter categories, we establish an Artin-type equivalence for sheaf categories over many sites in Section 3. Finally based on the intrinsic properties of (ringed) topoi, we continue to deduce some interesting consequences on linear representations. Our key results are still correct for finitely generated modules (in terms of coherent sheaves).\\

\noindent{\textbf{Acknowledgements}} The authors would like to thank the anonymous referee whose comments significantly improve the presentation of the paper.

\section{Sheaves and group representations}

We begin with basics of sheaves of sets on sites. Then we will continue to discuss sheaves with algebraic structures, and their applications to linear representations. This is heavy machinery, but we try to make it comprehensible for the reader.  Our main references on topos and sheaf theories are \cite{SGA4, St, Jo, MM}. However they use quite different notations and terminologies. We try to follow \cite{St, MM}, but choices and adjustments have to be made in order to keep consistency.

Modules will be \textit{right} modules in this paper, if not otherwise specified.

\subsection{Sites and sheaves} Let $\calC$ be a small category. We assume the reader is familiar with the concept of a functor category. For consistency, we shall call a contravariant functor on $\calC$ a \textit{presheaf} (also called a \textit{representation} of $\calC$, see \cite{We1}). We denote by $\PSh(\calC)$ the category of presheaves from $\calC$ to $\Set$. If $R$ is a commutative ring with identity, we write $\PSh(\calC,R)$ for the category of presheaves of $R$-modules (contravariant functors from $\calC$ to $\rMod{R}$). If $R=\ZZ$, this is often dubbed as $\PSh(\calC,\ZZ)=\PSh(\calC,\Ab)$. If $G$ is a group, regarded as a category with one object, $\PSh(G)$ and $\PSh(G,R)$ are canonically identified the the categories $\rSet{G}$ and $\rMod{RG}$, respectively.

One may put a \textit{Grothendieck topology} $\calJ_{\calC}$ on $\calC$, see \cite{MM}. 

By definition a \textit{sieve} $S$ on $x\in\Ob\calC$ is simply a subfunctor of $\Hom_{\calC}(-,x)$. In fact a sieve $S$ on $x$ is identified with a set (also written as $S$) of morphisms with codomain $x$, satisfying the condition that if ${\sf u}\in S$ and ${\sf uv}$ exists then ${\sf uv}\in S$. For instance the \textit{maximal sieve} $\Hom_{\calC}(-,x)$ on $x$ is given by the set of all morphisms with codomain $x$.

\begin{definition} Let $\calC$ be a small category. A \textit{Grothendieck topology} on $\calC$ is a function $\calJ_{\calC}$ which assigns to each object $x\in\Ob\calC$ a class of sieves $\calJ_{\calC}(x)$ on $x$, in such a way that
\begin{enumerate}
\item the maximal sieve $\Hom_{\calC}(-,x)$ is in $\calJ_{\calC}(x)$;
\item if $S\in\calJ_{\calC}(x)$, then ${\sf u}^*(S)=\{{\sf v}\bigm{|} {\sf uv}\in S\}$ lies in $\calJ_{\calC}(y)$ for any ${\sf u}:y\to x$;
\item if $S_1\in\calJ_{\calC}(x)$ and $S_2$ is any sieve on $x$ such that ${\sf u}^*(S_2)\in\calJ_{\calC}(y)$ for all ${\sf u}: y \to x$ in $S_1$, then $S_2\in\calJ_{\calC}(x)$.
\end{enumerate}
\end{definition}

Any sieve in $\calJ_{\calC}(x)$ is called a \textit{covering sieve} on $x$. A small category $\calC$ equipped with a Grothendieck topology $\calJ_{\calC}$ is called a \textit{site} $\C=(\calC,\calJ_{\calC})$ (the subscript is often omitted when there is no confusion).

Roughly speaking, a sheaf of sets on $\calC$ is a presheaf of sets on $\calC$, satisfying certain glueing properties mandated by $\calJ_{\calC}$. A concise definition of a sheaf is the following.

\begin{definition}
A presheaf $\frakF\in\PSh(\calC)$ is a ($\calJ$-)sheaf if, for every $x\in\Ob\calC$ and every $S\in\calJ(x)$, the inclusion $S\hookrightarrow\Hom_{\calC}(-,x)$ induces an isomorphism
$$
\Nat(\Hom_{\calC}(-,x),\frakF){\buildrel{\cong}\over{\to}}\Nat(S,\frakF).
$$
The category of sheaves of sets on $\C=(\calC,\calJ)$ is the full subcategory of $\PSh(\calC)$, consisting all all ($\calJ$-)sheaves, denoted by $\Sh(\C)$.
\end{definition}

Note that by Yoneda Lemma, $\frakF(x)\cong\Nat(\Hom_{\calC}(-,x),\frakF)$. 

In the literature on Grothendieck topologies, many authors ask for the existence of \textit{fibre products} in $\calC$, see for example \cite{Mi}. Although this condition is (essentially) met in our situation, we do not make the assumption, as we want to work with general finite categories.

\subsection{Examples and motivations}

We list some motivating examples, and state the result of M. Artin mentioned earlier.

\begin{example}
Given any category $\calC$, one can define the \textit{trivial topology} $\calJ_{triv}$ such that $\calJ_{triv}(x)=\{\Hom_{\calC}(-,x)\}, \forall x\in\Ob\calC$. Since this topology does not impose any extra condition, according to Definition 1.2, all presheaves are sheaves. 

A more interesting example is the \textit{atomic topology}, \cite[Section III.2]{MM}. This needs a mild assumption on $\calC$ : the Ore condition, weaker than having fibre products, which says that any two morphisms $y\to x$ and $z\to x$, in $\Mor\calC$, can be completed to a commutative square
$$
\xymatrix{w \ar[d] \ar[r] & z \ar[d]\\
y \ar[r] & x }
$$ 
and which is met by the categories we consider here. Then we put 
$$
\calJ_{at}(x)=\{\mbox{all non-empty sieves on}\ x\}, \forall x\in\Ob\calC.
$$
A sieve $S$, as a subfunctor of $\Hom_{\calC}(-,x)$, is non-empty if there is at least one object $y$ such that $S(y)\ne\emptyset$.
\end{example}

In general, $\Sh(\C)$ is a proper full subcategory of $\PSh(\calC)$, see Proposition 3.5. However if $\calC$ is given the trivial topology $\calJ_{triv}$, then $\Sh(\C)=\PSh(\calC)$. To some extent, the Grothendieck topologies allow us to add a new variance $\calJ$ on a category $\calC$ and to specify finer structural information on its representations.

\begin{remark}
One can continue to consider sheaves with algebraic structures, such as sheaves of $R$-modules for a given ring $R$ and sheaves of rings. By definition, a presheaf of $R$-modules is a sheaf if and only if, considered as a presheaf of sets, it is a sheaf. Sheaves of $R$-modules on a site $\C$ form a category $\Sh(\C,R)$, see \cite[7.44]{St}, which is abelian. This category has enough injectives, but usually not enough projectives.

As our strategy, we will first prove results about topoi (categories of sheaves of sets). Then there are ways to obtain corresponding statements about sheaves with algebraic structures \cite[7.43 \& 18.13]{St}.
\end{remark}

Since we are interested in group representations, we convert module categories into sheaf categories.

\begin{example} Let $G$ be a group, regarded as a category with one object $\bullet$ . Then $\PSh(G)$ is identified with $\rSet{G}$.

The atomic topology on $G$ is $\calJ_{at}(\bullet)=\{\Hom_G(-,\bullet)\}=\{G\}=\calJ_{triv}(\bullet)$. The resulting site is written as $\GG=(G,\calJ_{at})=(G,\calJ_{triv})$. Under the circumstance, every presheaf on $\GG$ is a sheaf, and we have $\Sh(\GG)=\PSh(G)$. It follows immediately that, for a given coefficient ring $R$, $\Sh(\GG, R)=\rMod{RG}$. Thus group representation theory is a sheaf theory.
\end{example}

Deligne-Lusztig theory successfully applies sheaf theory and \'etale cohomology theory to the representatiton theory of finite groups of Lie type. The more recent work of Balmer \cite{Ba} demonstrates a fuller scope of sheaf theory, aiming at broader applications to representations of general finite groups. Our work is originally prompted by Balmer's, but eventually pinned down to the following theorem.

\begin{theorem}[Artin, Mac Lane-Moerdijk] Let $G$ be a finite group and $\calO(G)$ be the orbit category. Then, on the site $\OG=(\calO(G),\calJ_{at})$,
$$
\Sh(\OG)\simeq\rSet{G}.
$$
Consequently, $\Sh(\OG,R)\simeq\rMod{RG}$, where $R$ is a commutative ring with identity.
\end{theorem}

\begin{proof} For future reference, we sketch the constructions of the two-way functors. For full details, see for instance \cite[I.0]{Ar} and \cite[III.9]{MM}. 

If $M$ is a $G$-set, one defines a presheaf $\frakF_M$ such that $\frakF_M(H\backslash G)=\Hom_G(H\backslash G, M)\cong M^H$. Then one can continue to verify it is a sheaf. Conversely, if $\frakF$ is a sheaf, $\frakF(1\backslash G)$ is naturally a $G$-set.
\end{proof}

\begin{remark} 
In \cite{Ar, MM, Mi}, the authors mainly dealt with (infinite) topological groups, continuous $G$-sets, and discrete $G$-modules. Their version applies to any topological group and to more general orbit categories containing a cofinal family of orbits.
\end{remark}

Let $\mathbf{B}G$ (not the classifying space $BG$) be the site of $\rSet{G}$ with the canonical topology. The above theorem is originally proved for $\mathbf{B}G$ \cite{Ar}, that is, $\Sh(\mathbf{B}G)\simeq\rSet{G}$. The present form is found in \cite[III.9]{MM}. In fact, the canonical topology restricts to the atomic topology on $\calO(G)$ (written as $\mathbf{S}(G)$ there). By the Comparison Lemma \cite[Appendix 4.3]{MM}, one has $\Sh(\mathbf{B}G)\simeq\Sh(\OG)$. Another way to understand this reduction is that $\rSet{G}$ is the \textit{additive extension} of $\calO(G)$, in the sense that every $G$-set is a disjoint union of objects of $\calO(G)$, up to isomorphism.  Thus every presheaf (resp. sheaf) on $\rSet{G}$ restricts to a presheaf (resp. sheaf) on $\calO(G)$. While on the other hand, every presheaf (resp. sheaf) on $\calO(G)$ extends to a presheaf (resp. sheaf) on $\rSet{G}$.

We prefer to work with $\calO(G)$ because it is finite. The category of presheaves of $R$-modules is equivalent to the module category of $R\calO(G)$, the orbit category algebra \cite{AKO, Li1, We3}, which is finite-dimensional. The sheaf $\frakF_M$ is also called the \textit{fixed-point functor or (pre)sheaf}, which is a Mackey functor \cite{We1} and which is key to the Ronan-Smith theory \cite[Ch. 10]{Sm}, when $M$ is a finitely generated $RG$-module. 

\section{Morphisms of sites and topoi} 

Although the category equivalence in Theorem 1.5 is neatly presented, it is not coming from functors between $\calO(G)$ and $G$. For better understanding, we insert a third category, the \textit{transporter category} $\calT(G)$, and demonstrate that the equivalence is truly induced by functors among $\calO(G), \calT(G)$ and $G$. Thus the equivalence can be considered as a generalized restriction and/or a generalized induction, in the sense of \cite{Xu1}. To this end, we need to compare sheaf categories.

For future reference, we recall some fundamental concepts and constructions, found in many places, see for instance \cite{St, KS}. Given a (covariant) functor $\alpha:\calC\to\calD$, there exists a \textit{restriction} along $\alpha$, $\Res_{\alpha} : \PSh(\calD)\to\PSh(\calC)$ which is given by the precomposition with $\alpha$. This functor admits two adjoint functors, the left and right \textit{Kan extensions} along $\alpha$, written as $LK_{\alpha}, RK_{\alpha} : \PSh(\calC)\to\PSh(\calD)$. Let $d\in\Ob\calD$. Then the \textit{comma category} (or \textit{category under} $d$, or just \textit{undercategory}) $d/\alpha$ has objects $\{({\sf t},x)\bigm{|} {\sf t}: d\to\alpha(x), x\in\Ob\calC\}$. A morphism $({\sf t},x)\to({\sf t}',x')$ is given by ${\sf u}:x\to x'$ such that ${\sf t}'=\alpha({\sf u}){\sf t}$. Comma categories (including undercategories and overcategories) are used to define the Kan extensions.

To compare sheaves on two sites, we need the concepts of continuous and cocontinuous functors \cite{St, Jo}. Recall that given a set of morphisms $\{{\sf u}_i\}_{i\in I}$ with common codoamin $x\in\Ob\calC$, the sieve $S$ generated by $\{{\sf u}_i\}_{i\in I}$ is the smallest sieve that contains these ${\sf u}_i$'s, written as $S=({\sf u}_i | i\in I)$.

\begin{definition} Let $\C=(\calC,\calJ_{\calC})$ and $\D=(\calD,\calJ_{\calD})$ be two sites. A functor $\alpha: \calC\to\calD$ is called \textit{continuous} if, for every $x\in\Ob\calC$ and every $S_x\in\calJ_{\calC}(x)$, $\alpha(S_x)$ generates a covering sieve on $\alpha(x)$ and, for every $d\in\Ob\calD$, $(d/\alpha)^{op}$ is filtered. 

A functor $\beta : \calC\to\calD$ is called \textit{cocontinuous} if for every $x\in\Ob\calC$ and every covering sieve $S_{\beta(x)}\in\calJ_{\calD}(\beta(x))$ on $\beta(x)\in\Ob\calD$, there exists a covering sieve $S_x\in\calJ_{\calC}(x)$ with $\beta(S_x)\subset S_{\beta(x)}$. (This is equivalent to saying that $\{{\sf u}\in\Hom_{\calC}(-,x)\bigm{|}\beta({\sf u})\in S_{\beta(x)}\}$ is a covering sieve on $x$.)
\end{definition}

The original definition of continuity simply asks $\Res_{\alpha}$ to preserve sheaves \cite{SGA4}. Our definition is taken from \cite[C2.3.7]{Jo}, see \cite[C2.3.1]{Jo} and \cite[7.13.1]{St} for a version based on fibre products. The first condition on continuity is called \textit{cover-preserving} and the second is said to be \textit{flat}. A cocontinuous functor is also called \textit{cover-reflecting}. The definition of cocontinuity is \cite[Page 574]{Jo}, see also \cite[7.19.1]{St}. For those who worry about fibre products, see Remark 3.2 for the relevant constructions for the sites we have in mind. The following is \cite[7.14.1]{St}, see also \cite[C2.3.1]{Jo}

\begin{definition} Let $\C=(\calC,\calJ_{\calC})$ and $\D=(\calD,\calJ_{\calD})$ be two sites. A \textit{morphism of sites} $\Lambda : \D\to\C$ is given by a continuous functor $\alpha : \calC\to\calD$ such that $LK_{\alpha}^{\sharp} : \Sh(\C)\to\Sh(\D)$ is exact.
\end{definition}

Here for a sheaf $\frakG$ on $\C$, $LK_{\alpha}^{\sharp}\frakG:=(LK_{\alpha}\frakG)^{\sharp}$, where $(-)^\sharp$ means \textit{sheafification}. See Proposition 3.5 where we recall the construction of a sheafification in the proof. The sheafification functor is exact and is left adjoint to the forgetful functor $\Sh(\C)\to\PSh(\calC)$, which is left exact. 

The following is \cite[7.15.1]{St}, see also \cite[A4.1.1]{Jo} and \cite[VII.1]{MM}.

\begin{definition} A \textit{(Grothendieck) topos} is a category that is equivalent to some $\Sh(\C)$, a sheaf category of sets on a site $\C$. A \textit{(geometric) morphism of topoi} 
$$
\Psi=(\Psi^{-1},\Psi_*) : \Sh(\D)\to\Sh(\C)
$$ 
consists of a pair of functors $\Psi_* : \Sh(\D)\to\Sh(\C)$ and $\Psi^{-1}:\Sh(\C)\to\Sh(\D)$ such that $\Psi^{-1}$ is left exact and is left adjoint to $\Psi_*$.
\end{definition}

A morphism of sites $\Lambda : \D\to\C$ always induces a morphism of topoi $\Lambda=(\Lambda^{-1},\Lambda_*)$, where $\Lambda_*=\Res_{\alpha} : \Sh(\D)\to\Sh(\C)$ and $\Lambda^{-1}=LK_{\alpha}^{\sharp} : \Sh(\C)\to\Sh(\D)$. There is a broader concept of an \textit{elementary topos} \cite{Jo, MM}. We only work with Grothendieck topoi in this paper.

\begin{remark}
By \cite[7.21.1]{St} and \cite[C.2.3.18]{Jo}, if $\beta : \calC\to\calD$ is cocontinuous, then $RK_{\beta}\frakF\in\Sh(\D)$, for any $\frakF\in\Sh(\C)$. Moreover $\Res_{\beta}^{\sharp} : \Sh(\D)\to\Sh(\C)$ is exact and left adjoint to $RK_{\beta}$. The cocontinuous functor $\beta :\calC\to\calD$ also gives rise to a morphism of topoi
$$
\Phi=(\Phi^{-1},\Phi_*) : \Sh(\C) \to \Sh(\D),
$$
where $\Phi_*=RK_{\beta}$ and $\Phi^{-1}=\Res_{\beta}^{\sharp}$. 
\end{remark}

For future applications, we record the following (\cite[7.21.5]{St}, also see the comments before \cite[C.2.3.23]{Jo}).

\begin{lemma} Let $\C$ and $\D$ be two sites. If there is a continuous and cocontinuous functor $\beta:\calC\to\calD$, for the associated morphism of topoi $\Phi=(\Phi^{-1},\Phi_*) : \Sh(\C)\to\Sh(\D)$, we have
\begin{enumerate}
\item $\Phi^{-1}=\Res_{\beta}$,
\item $\Phi_!=LK_{\beta}^{\sharp}$ is left adjoint to $\Phi^{-1}$.
\end{enumerate}
\end{lemma}

Many results on topoi may be readily applied to sheaves with algebraic structures. This is vital for investigating linear representations. Once again, from \cite[7.43]{St} the aforementioned constructions and properties we know about sheaves of sets sustain on sheaves of abelian groups, $R$-modules etc., in the sense that the morphism $\Psi$ of topoi gives rise to a pair of adjoint functors $(\Psi^{-1},\Psi_*)$ on categories of sheaves with suitable algebraic structures.

\section{The transporter category} 

A transporter category is a bridge between a group and various local categories \cite{Xu1}. We shall see this connection leads us to the aforementioned Theorem 1.6, in a straightforward way. On these categories (equipped with atomic topology), we can carry on explicit calculations, which eventually become cornerstones of our main results. 

Let $\calP$ be a $G$-poset. For consistency, we put the $G$-action on the right. The group action permutes both $\Ob\calP$ and $\Mor\calP$. The morphisms in $\calP$ will be written as $\iota^y_x$ if $x\le y$. Given an object $x\in\Ob\calP$ and a morphism $\iota^y_x\in\Mor\calP$, we denote by $x^g$ and $(\iota^y_x)^g=\iota^{y^g}_{x^g}$ their images under the action by $g\in G$.

By definition, the (abstract) \textit{transporter category} $\PG$, constructed on a $G$-category $\calP$, is a special kind of \textit{Grothendieck construction} and has the same objects as $\calP$. Its morphisms are the formal products $\iota^y_xg:x^{g^{-1}}\to y$, where $\iota^y_x\in\Mor\calP$ and $g\in G$. For convenience, we rewrite morphisms in the form $\iota^y_{x^g} g: x \to y$. If $\iota^z_{w^h}h : w \to z$ is another morphism, then $(\iota^y_{x^g}g)(\iota^z_{w^h}h)$ exists in $\Mor\PG$ if and only if $\iota^y_{x^g}(\iota^z_{w^h})^g$ exists in $\calP$. The latter is equivalent to saying that $x=z$. The composite is $(\iota^y_{x^g}g)(\iota^x_{w^h}h):=\iota^y_{w^{hg}}(hg)$. Since $\iota^y_{x^g} g: x\to y$ itself factorizes as $(\iota^y_{x^g}1)(\iota^{x^g}_{x^g}g)$ by $\iota^{x^g}_{x^g}g:x\to x^g$ and $\iota^y_{x^g}1: x^g\to y$, $\iota^y_{x^g} g$ may be understood as a ``conjugation'' followed by an ``inclusion''.

Note that $\calP\rtimes G$ does not have to meet the Ore condition. Think of a poset with at least two minimal objects.

\begin{example} Let $H$ be a subgroup of $G$. The group $G$ acts on $H\backslash G$ by right multiplication. One can check that $(H\backslash G)\rtimes G$ is a connected groupoid whose skeleton can be identified with the subgroup $H$. Thus $(H\backslash G)\rtimes G\simeq H$. 

For a general $G$-poset $\calP$, $\PG$ can be regarded as a generalized subgroup of $G$, see \cite{Xu1}.
\end{example}

When $\calP=\calS(G)$ is the poset of all subgroups of $G$, equipped with the conjugation action, $\calT(G)=\calS(G)\rtimes G$ is called the (complete) transporter category. In fact, a morphism set can be characterized by
$$
\Hom_{\calT(G)}(H,K)=\{\iota^K_{H^g}g : H \to K \bigm{|} g^{-1}Hg\subset K\},
$$
where $H^g=g^{-1}Hg$. The orbit category $\calO(G)$ is a quotient category of $\calT(G)$, in the sense that there is a functor $\rho : \calT(G)\to\calO(G)$ with $\rho(H)=H\backslash G$ (bijection on objects) and $\rho(\iota^K_{H^g}g)=c_{g^{-1}}$ (surjection on morphism sets), where $c_{g^{-1}}: H\backslash G\to K\backslash G$ is given by $H\mapsto Kg^{-1}$. One easily checks that 
$$
\Hom_{\calO(G)}(H\backslash G, K\backslash G)\cong K\backslash\Hom_{\calT(G)}(H, K).
$$

\begin{remark}
The category $\calT(G)$ has fibre products, and thus satisfies the Ore condition. If $\iota^L_{H^g}g : H\to L$ and $\iota^L_{K^{g'}}g': K\to L$ are two morphisms in the transporter category, then the fibre product is simply $H^g\cap K^{g'}$ since by definition these two morphisms are injective homomorphisms. We put $\TG=(\calT(G),\calJ_{at})$.

A description of the fibre products in $\rSet{G}$ can be found in \cite[2.1]{We1}. Although the fibre products do not always exist in $\calO(G)$, it will not affect our upcoming discussions because of the equivalence $\Sh(\OG)\simeq\Sh(\mathbf{B}G)$. Here $\mathbf{B}G$ stands for the site on $\rSet{G}$ with canonical topology.

Finally in the single-object category $G$, one has the fibre product $\bullet\times_{\bullet}\bullet=\bullet$.
\end{remark}

Every transporter category $\PG$ admits a natural functor to $G$. The functor $\pi$ send every objects to $\bullet$ and each morphism ${\sf u}g$ to $g$. Meanwhile, a transporter category is frequently encoded into a category extension \cite{We2, Xu1}. By definition, a category extension (generalizing the notion of a group extension) is a sequence of two functors
$$
\calK {\buildrel{i}\over{\to}} \calE {\buildrel{\rho}\over{\to}} \calC
$$
such that
\begin{enumerate}
\item three categories have the same objects, $i$ and $\rho$ are identities on objects, $i$ is injective on morphisms and $\rho$ is surjective on objects,

\item whenever there are ${\sf u, v}\in\Mor\calE$ with $\rho({\sf u})=\rho({\sf v})$, there is a unique ${\sf w}\in\Mor\calK$ satisfying ${\sf u}\iota({\sf w})={\sf v}$.
\end{enumerate}
It implies that $\calK=\coprod_{x\in\Ob\calC}\calK(x)$, where, for each $x \in \Ob\calC$, $\calK(x)$ is a category whose only object is $x$ and whose morphisms are those mapped to $1_x$ by $\rho i$. Morphisms of $\calK(x)$ must be invertible and thus $\calK(x)$ is essentially a group.  If $\calK(x)$ is always trivial, then $\calC\cong\PG$.

Suppose we have a category extension $\calK \to \PG \to \calC$, fitting into the following picture
$$
\xymatrix{ & & \calK \ar[dl]_{i}\\
&\PG \ar[dl]_{\rho} \ar[dr]^{\pi}&\\
\calC && G}
$$
where $\calK(x)\subset\Aut_{\PG}(x)$ is a subgroup, for each $x \in \Ob\PG=\Ob\calC$, and $\calK=\coprod_{x\in\Ob\calC}\calK(x)$.

These two functors induce several pairs of adjoint functors
$$
\xymatrix{& \PSh(\PG) \ar@/_/[dl]_{RK_{\rho},LK_{\rho}} \ar@/^/[dr]^{LK_{\pi}, RK_{\pi}}&\\
\PSh(\calC) \ar@/_/[ur]_{\Res_{\rho}} && \PSh(G)
\ar@/^/[ul]^{\Res_{\pi}}.}
$$
We already know $\PSh(G)=\rSet{G}$. Given a $G$-set $M$, the restriction along $\pi$ sends it to the constant presheaf $\kappa_M=\Res_{\pi}(M)\in\PSh(\PG)$. Be aware that $\kappa_M$ is only constant on objects, not constant on morphisms (which are induced by the $G$-action on $M$). The Kan extensions are computed in \cite{Xu1}. The left and right Kan extensions along $\pi$ are isomorphic to $\lim_{\calP}$ and $\colim_{\calP}$, respectively. Meanwhile the right Kan extension along $\rho$ is given by $(RK_{\rho}\mathfrak{F})(x)=\mathfrak{F}(x)^{\calK(x)}$, $\forall x\in\Ob\calC$ and $\mathfrak{F}\in\PSh(\PG)$.

We want to deduce an analogous picture for sheaf categories. To this end, we shall classify the sheaves on a certain type of sites defined over $\PG$. The following key lemma will allow us to compute the sheafification of a presheaf.

\begin{lemma} Let $\calP$ be a $G$-poset with an initial object. Then $\calP\rtimes G$ satisfies the Ore condition and can be given the atomic topology. Meanwhile on each $x\in\Ob\calP\rtimes G$, there is a unique minimal non-empty sieve.
\end{lemma}

\begin{proof} Set $x_0$ as the (unique) initial object in $\calP$. It is fixed under the $G$-action, that is, $x_0^g=x_0, \forall g\in G$. Let $\iota^x_{y^g}g: y \to x$ and $\iota^x_{z^h}h: z \to x$ be two morphisms. Then $\iota^y_{x_0}g^{-1} : x_0\to y$ and $\iota^z_{x_0}h^{-1} : x_0 \to z$ will complete the square in the Ore condition. Thus one can introduce the atomic topology on $\calP\rtimes G$.

Suppose $S$ is a non-empty sieve on $x$ with a morphism $\iota^x_{y^g}g : y \to x$ in $\Mor\calP\rtimes G$.  Consider the morphism $\iota^y_{x_0}g^{-1} : x_0 \to y$. The composite
$(\iota^x_{y^g}g)(\iota^y_{x_0}g^{-1}) : x_0 \to x$ belongs to $S$. Since the composite is $\iota^x_{x_0} 1\in S$, $(\iota^x_{x_0} 1)(\iota^{x_0}_{x_0}g')=\iota^x_{x_0}g'\in S$ for all $g'\in G$. One immediately verifies that $S_x^{\min}=\{\iota^x_{x_0}g'\bigm{|} g'\in G\}$ is a sieve on $x$, and it is minimal and unique among non-empty sieves on $x$. 

In terms of a subfunctor of $S_x^{\max}=\Hom_{\calP\rtimes G}(-,x)$, $S_x^{\min}$ is characterized by 
$$
S_x^{\min}(y) = \left\{ \begin{array}{ll}
              \Hom_{\PG}(x_0,x) ,& \mbox{if}\ y=x_0;\\
              \emptyset\   ,& \mbox{otherwise.}
             \end{array}
      \right.
$$
\end{proof}

Note that $\Hom_{\PG}(x_0,x)$ can be identified with $G$, as a right $G$-set. For future reference, we set $\pG=(\PG,\calJ_{at})$ whenever $\calP$ has an initial object. Then $\TG$ is a special case of $\pG$.

The above properties are inherited by quotient categories.

\begin{corollary} Let $\calP$ be a $G$-poset with an initial object. Suppose we have a covariant functor $\rho : \PG \to \calC$, which is identity on objects and which is surjective on morphism sets. Then $\calC$ satisfies the Ore condition and can be given the atomic topology. There is a unique minimal non-empty sieve on each $x\in\Ob\calC$.
\end{corollary}

\begin{proof} For convenience, we identity the objects of $\PG$ and $\calC$. Take two morphisms $y\to x$ and $z \to x$ in $\calC$. Choose a preimage for each of them in $\PG$. These morphisms are still with codomain $x$ and can be completed to a commutative square. The image of the square provides a commutative square in $\calC$.

Let $S$ be a non-empty sieve on $x\in\Ob\calC$. Then 
$$
\rho^{-1}(S)=\{\alpha\in\Mor\PG \bigm{|} \rho(\alpha)\in S\subset\Mor\calC\}
$$ 
is a non-empty sieve on $x\in\Ob\PG$. Meanwhile if $S'$ is a non-empty sieve on $x\in\Ob\PG$, $\rho(S')$ becomes a non-empty sieve on $x\in\Ob\calC$. Therefore we have $S_x^{\min}\subset\rho^{-1}(S)$ and $\rho(S_x^{\min})\subset S$. One immediately sees that $\rho(S_x^{\min})$ has to be minimal within non-empty sieves on $x\in\Ob\calC$, and it is unique.

In terms of a subfunctor of $S_x^{\max}=\Hom_{\calC}(-,x)$, it is characterized by 
$$
S_x^{\min}(y) = \left\{ \begin{array}{ll}
              \Hom_{\calC}(x_0,x) ,& \mbox{if}\ y=x_0;\\
              \emptyset\   ,& \mbox{otherwise.}
             \end{array}
      \right.
$$
\end{proof}

If $\calC=\PG$ and $\rho$ is the identity functor, then we go back to the situation in Lemma 3.3.

\begin{proposition} Let $\calP$ be a $G$-poset with an initial object $x_0$, and suppose the following is a category extension (see \cite{We1, Xu1})
$$
\calK \to \PG\ {\buildrel{\rho}\over{\to}}\ \calC.
$$
Given the atomic topology on $\calC$, the sheafification of $\frakG\in\PSh(\calC)$ is the fixed-point sheaf $\frakF_{\frakG(x_0)}\in\Sh(\C)$ in the sense that $\frakF_{\frakG(x_0)}(x)=\frakG(x_0)^{\calK(x)}$. 

Consequently sheaves on $\C$ are of the form $\frakF_M$ with $M$ running over all $G$-sets. Particulary the sheaves on $\pG$ are all the constant sheaves $\kappa_M$.
\end{proposition}

\begin{proof} By standard procedures of construction (see for instance \cite[III.5]{MM} or \cite[7.10]{St}), given a presheaf $\frakG$, we shall first construct $\frakG^{\dagger}\in\PSh(\calC)$ which is given by
$$
\frakG^{\dagger}(x)=\lim_{S\in\calJ(x)}\Nat(S,\frakG), \forall x\in\Ob\calC.
$$ 
If $\frakG^{\dagger}$ is not a sheaf, then we repeat the construction and $\frakG^{\sharp}=(\frakG^{\dagger})^{\dagger} \in\Sh(\C)$ is the sheafification of $\frakG$.

By Corollary 3.4, on every $x\in\Ob\calC$ there is always a unique minimal non-empty sieve $S_x^{\min}\in\calJ(x)$ characterized by 
$$
S_x^{\min}(y) = \left\{ \begin{array}{ll}
              \Hom_{\calC}(x_0,x) ,& \mbox{if}\ y=x_0;\\
              \emptyset\   ,& \mbox{otherwise}.
             \end{array}
      \right.
$$
Moreover $\Hom_{\calC}(x_0,x)=\calK(x)\backslash\Hom_{\PG}(x_0,x)$ is identified with the right $G$-set $\calK(x)\backslash G$. Thus we find $\Nat(S_x^{\min},\frakG)$ is isomorphic to
$$
\Hom_G(\calK(x)\backslash G,\frakG(x_0))\cong \frakG(x_0)^{\calK(x)}.
$$
It implies that $\frakG^{\dagger}(x)\cong \frakG(x_0)^{\calK(x)}$ for each $x\in\Ob\calC$ and it is straightforward to verify that $\frakG^{\dagger}\cong\frakF_{\frakG(x_0)}$. However $\frakF_{\frakG(x_0)}^{\dagger}=\frakF_{\frakG(x_0)}$ is already a sheaf, that is, $\frakG^{\sharp}\cong\frakF_{\frakG(x_0)}$.

Note that $\frakG(x_0)$ is always a $G$-set. Moreover if $M$ is a $G$-set, we can define a presheaf $\frakG_M$ such that $\frakG(x_0)=M$ and $\frakG(x)=\emptyset$, whenever $x\ne x_0$. Its sheafification is simply $\frakF_M$. Particularly if $\calC=\PG$, then $\calK(x)=1$ for all objects. Under the circumstance, we have $\frakG^{\sharp}\cong\kappa_{\frakG(x_0)}$, a constant sheaf.
\end{proof}

In order to establish functors between sheaf categories, we discuss the continuity and/or cocontinuity of $\pi$ and $\rho$. 

\begin{proposition} Suppose $\calP$ is a $G$-poset with an initial object. Let $\pi:\PG\to G$ and $\rho: \PG\to\calC$ be the above two functors. Then, under atomic topologies on three categories,
\begin{enumerate}
\item $\pi$ is both continuous and cocontinuous;
\item $\rho$ is cocontinuous.
\end{enumerate}
\end{proposition}

Remark 3.2 hints that when $\calP=\calS(G)$ $\rho$ will not preserve fibre products, regarded as a functor $\calT(G)\to\rSet{G}$. Hence it is not continuous.

\begin{proof} To verify that $\pi$ is continuous, we first note that, for any $S_x\in\calJ(x)$, $\pi(S_x)$ generates exactly the only covering sieve on $\bullet$. Then since $\bullet\backslash\pi$ has the poset $\calP$ as a skeleton \cite{Xu1}, its opposite has a terminal object and thus is filtered. 

To show $\pi$ is cocontinuous, we need to demonstrate that for each object $x\in\Ob\PG$ and the unique covering sieve $S\in\calJ(\pi(x))$, the maps $y \to x$ such that $\pi(y) \to \pi(x)$ factors through $S$ is a covering sieve on $x$. These maps include all those with codomain $x$, and form the maximal sieve which covers $x$.

In a similar fashion, one establishes the cocontinuity of $\rho$.
\end{proof}

One can also use \cite[III.4.2]{MM} to verify directly that $\kappa_M=\Res_{\pi}(M)$ is a sheaf, for all $M\in\PSh(G)=\Sh(\GG)$. 

We note that $\pi$ usually does not give rise to a morphism of sites, for 
$$
\Sh(\pG)\hookrightarrow\PSh(\PG) {\buildrel{LK_{\pi}}\over{\longrightarrow}} \PSh(G)=\Sh(\GG)
$$ 
is not exact. Nonetheless we have two morphisms of topoi.

\begin{corollary} Suppose $\calP$ is a $G$-poset with an initial object. The functor $\pi : \PG \to G$ induces a morphism of topoi
$$
\Xi=(\Xi^{-1},\Xi_*) : \Sh(\GG) \to \Sh(\pG).
$$
as well as another morphism of topoi
$$
\Pi=(\Pi^{-1},\Pi_*) : \Sh(\pG) \to \Sh(\GG).
$$
The functor $\rho:\PG \to\calC$ induces a morphism of topoi
$$
\Theta=(\Theta^{-1},\Theta_*) : \Sh(\pG) \to \Sh(\C).
$$
\end{corollary}

\begin{proof} Define $\Xi_*=\Res_{\pi} : \Sh(\GG) \to \Sh(\pG)$ and 
$$
\Xi^{-1} : \Sh(\pG)\hookrightarrow\PSh(\PG) {\buildrel{LK_{\pi}}\over{\longrightarrow}} \PSh(G)=\Sh(\GG)
$$
It is crucial to see that $\Xi^{-1}$ is left exact. In fact, The forgetful functor is left exact. Meanwhile $LK_{\pi}=LK_{\pi}^{\sharp}$ is exact since $LK_{\pi}\cong\lim_{\calP}$ and $\calP$ has an initial object $x_0$. Note that $\Xi^{-1}(\frakF)=\frakF(x_0)$.

The cocontinuity of $\pi$ says that $RK_{\pi}\frakF=\colim_{\calP}\frakF\in\Sh(\GG)$, for any $\frakF\in\Sh(\pG)$. Then we put $\Pi_*=RK_{\pi}$ and $\Pi^{-1}=\Res_{\pi}$.

The cocontinuity of $\rho$ implies that $RK_{\rho}\frakF\in\Sh(\C)$, for any $\frakF\in\Sh(\pG)$.  We set $\Theta_*=RK_{\rho}$ and  $\Theta^{-1}=\Res_{\rho}^{\sharp}$. While for each $\frakG\in\Sh(\C)$, $\Res_{\rho}^{\sharp}\frakG:=(\Res_{\rho}\frakG)^{\sharp}$. The functor $\Theta^{-1}$ is exact.
\end{proof}

Suppose $\calP$ is a $G$-poset with an initial object. By Proposition 3.5, the sheaves on $\pG$ are the constant sheaves $\kappa_M$ and those on $\C$ are all the fixed-point sheaves $\frakF_M$, with $M$ varying over all $G$-sets. In summary we have the following adjoint functors between sheaf categories
$$
\xymatrix{& \Sh(\pG) \ar@/_/[dl]_{\Theta_*} \ar@/^/[dr]^{\Pi_*,\ \Xi^{-1}} &\\
\Sh(\C) \ar@/_/[ur]_{\Theta^{-1}} && \Sh(\GG)
\ar@/^/[ul]^{\Pi^{-1},\ \Xi_*}.}
$$
All the constructions are intrinsic to $G$, and by the characterizations of sheaves, these functors give equivalences. (Since $\Sh(\pG)$ consists of constant sheaves only, we have $LK_{\pi}=RK_{\pi}$ on them.)

It implies that $\Xi, \Pi$ and $\Theta$ are equivalences of topoi.

\begin{theorem} Suppose $\calP$ is a $G$-poset with an initial object, and $\PG\to\calC$ is a category extension.
\begin{enumerate}
\item Both $\Sh(\GG)$ and $\Sh(\C)$ are equivalent to $\Sh(\pG)$, given by $\Xi$ and $\Theta$.

\item Let $\Upsilon=\Theta\Xi$. The composed morphism between topoi 
$$
\Upsilon=(\Upsilon^{-1},\Upsilon_*) : \Sh(\GG)\to\Sh(\C)
$$ 
is an equivalence, with $\Upsilon_*=RK_{\rho}\Res_{\pi}$ and $\Upsilon^{-1}=LK_{\pi}\Res_{\rho}^{\sharp}$.
\end{enumerate}
\end{theorem}

\begin{proof} The first part follows from Corollary 3.4 and Proposition 3.5 (and comments afterwards). Direct calculations show that they together provide a topos equivalence. The second follows from Proposition 3.5.
\end{proof}

Indeed, when $\calP=\calS(G)$, we can choose $\calC=\calO(G)$ (upon identifying objects of $\calT(G)=\calS(G)\rtimes G$ and $\calO(G)$). Then $\Upsilon_*$ and $\Upsilon^{-1}$ specialize to the two functors constructed in Theorem 1.6.

We may apply the result to the following posets (with conjugation action by $G$), which are built on various collections \cite[7.6]{DH} and \cite[Sections 37 \& 40]{Th} of subgroups of $G$ 
\vspace{3mm}
\begin{center}
\begin{tabular}{l|l}
\hline
posets $\calP$ & collections ($k$ being a field of characteristic $p$)\\
\hline
$\calS(G)$ & all subgroups \\
$\calS_p(G)$ & all $p$-subgroups for a prime $p$ \\
$\calS_b(G)$ & all $b$-Brauer pairs for a $p$-block $b$ of $kG$\\
$\calP(G)$ & all local pointed groups over $kGb$\\
\hline
\end{tabular}
\end{center}
\vspace{3mm}
In terms of homology decompositions of classifying spaces, the above collections are called trivial \cite[8.6]{DH} as they are contractible. The resulting transporter categories are as follows.

\begin{center}
\begin{tabular}{rl|ll}
\hline
& $\PG$ & $\calC$ &\\
\hline
transporter category & $\calT(G)$ & $\calO(G)$ & orbit category\\
$p$-transporter category & $\calT_p(G)$ & $\calO_p(G)$ & $p$-orbit category\\
$b$-transporter category & $\calT_b(G)$ & $\calO_b(G)$ & $b$-orbit category\\
$p$-local transporter category & $\calT_{LP}(G)$ & $\calO_{LP}(G)$ & $p$-local orbit category\\
\hline
\end{tabular}
\end{center}
\vspace{3mm}
Here in the category extension $\calK\to\PG\to\calC$, $\calK$ consists of all the subgroups in the corresponding collection. On each subgroup $H$ in a collection, its automorphism groups in $\calK, \PG$ and $\calC$, respectively, form a group extension $H \to N_G(H) \to N_G(H)/H$. In either case, the functor $\pi$ induces a homotopy equivalence between the classifying spaces of $\PG$ and $G$ \cite[8.6]{DH},
$$
B(\PG)\simeq {\rm hocolim}_GB\calP\simeq {\rm hocolim}_G\bullet \simeq BG.
$$ 
Below may be regarded as a representation-theoretic counterpart of this.

\begin{corollary} There are natural category equivalences among 
$$
\Sh(\GG), \Sh(\TG), \Sh(\OG), \Sh(\mathbf{T}_pG), \Sh(\mathbf{O}_pG)
$$
for any prime $p$.
\end{corollary}

Since all the categories and functors are constructed intrinsically from $G$, our proof reveals more information than that found in \cite{Ar, MM}. A natural question to ask is that what else posets may be used to obtain category equivalences, or something close? The poset $\calS^{\circ}_p(G)$ of \textit{non-trivial} $p$-subgroups stands out prominently, see \cite[8.7-8.10]{DH}. It seems that very different techniques are needed to deal with it.

The local categories of \cite[Section 47]{Th} form \textit{opposite extensions} with certain transporter categories \cite{Xu1}. Their sheaf categories also require different treatments.

\begin{remark} If $\calP$ is a $G$-poset with a terminal object, then we can still obtain some category equivalences. The reason is that $(\PG)^{op}\cong\calP^{op}\rtimes G^{op}$ while $\calP^{op}$ has an initial object and $G^{op}\cong G$.
\end{remark}

The results in this section are still true for sheaves with suitable algebraic structures. We make these explicit in the next section.

\section{Modules on sites}

In finite group representations, we often need to work with finitely generated modules. The corresponding concept is a coherent sheaf on a ringed site. To this end, we shall recall some more definitions from \cite{St}. 

\subsection{Ringed sites and ringed topoi} For the reader's convenience, we copy the definitions from \cite{St}. 

\begin{definition} 
\begin{enumerate}
\item A \textit{ringed site} is a pair $(\C,\calR_{\C})$ where $\C$ is a site and $\calR_{\C}$ is a sheaf of rings on $\C$. The sheaf $\calR_{\C}$ is called the \textit{structure sheaf} of the ringed site.

\item A \textit{ringed topos} is a pair $(\Sh(\C),\calR_{\C})$ where $\C$ is a site and $\calR_{\C}$ is a sheaf of rings on $\C$. The sheaf $\calR_{\C}$ is the \textit{structure sheaf} of the ringed site.

\item Let $(\Sh(\C),\calR_{\C}), (\Sh(\C'),\calR_{\C'})$ be ringed topoi. A \textit{morphism of ringed topoi} $(\Psi, \Psi^{\sharp}) : (\Sh(\C),\calR_{\C}) \to (\Sh(\C'),\calR_{\C'})$ is given by a morphism of topoi $\Psi : \Sh(\C) \to \Sh(\C')$ together with a map of sheaves of rings $\Psi^{\sharp} : \Psi^{-1}\calR_{\C'}\to\calR_{\C}$, which by adjunction is the same thing as a map of sheaves of rings $\Psi^{\sharp} : \calR_{\C'}\to \Psi_*(\calR_{\C})$.
\end{enumerate}
\end{definition}

Note that here $\Psi^{\sharp}$ is \textit{not} a sheafification, as $\Psi$ is not a sheaf! It might be an unfortunate choice of symbol, but it wouldn't cause much confusion.

We shall make $\GG$ into a ringed site with the structure sheaf $\calR_{\GG}$ such that $\calR_{\GG}(\bullet)=R$. Suppose $\calP$ is a $G$-poset with an initial object and $\PG\to\calC$ is a category extension. For $\pG$ and $\C$, they are ringed sites with the structure sheaves $\calR_{\pG}=\kappa_R$ and $\calR_{\C}=\frakF_R$. All of these are constant sheaves of commutative rings.

\begin{lemma} Suppose $\calP$ is a $G$-poset with an initial object and $\PG\to\calC$ is a category extension. The functors $\pi: \PG \to G$ and $\rho : \PG \to \calC$ induces morphisms of ringed topoi
$$
(\Xi,\Xi^{\sharp}) : (\Sh(\GG),\calR_{\GG}) \to (\Sh(\pG),\calR_{\pG})
$$
with $\Xi^{\sharp} : \Xi^{-1}\calR_{\pG}\to\calR_{\GG}$ being identity, and
$$
(\Theta,\Theta^{\sharp}) : (\Sh(\pG),\calR_{\pG}) \to (\Sh(\C),\calR_{\C})
$$
with $\Theta^{\sharp} : \Theta^{-1}\calR_{\C} \to \calR_{\pG}$ being identity.
\end{lemma}

\begin{proof} These follow from Corollary 3.7.
\end{proof}

\subsection{Modules on sites} Since we want to understand group representations through the equivalence
$$
\Sh(\C,R)\simeq\Sh(\GG,R)\cong\rMod{RG},
$$
we are keen to characterize finitely generated modules by sheaves.

\begin{definition} Let $(\C,\calR_{\C})$ be a ringed site.
\begin{enumerate}
\item A presheaf of $\calR_{\C}$-modules is given by an abelian presheaf $\frakF$ together with a map of presheaves of sets $\frakF \times \calR_{\C} \to \frakF$ such that, for every object $x$ of $\C$, the map $\frakF(x) \times \calR_{\C}(x) \to \frakF(x)$ defines
the structure of an $\calR_{\C}(x)$-module structure on the abelian group $\frakF(x)$.
\item A sheaf of $\calR_{\C}$-modules is a presheaf of $\calR_{\C}$-modules $\frakF$, such that the underlying presheaf of abelian groups $\frakF$ is a sheaf.
\item A morphism of sheaves of $\calR_{\C}$-modules is a morphism of presheaves of $\calR_{\C}$-modules.
\end{enumerate}
The category of sheaves of $\calR_{\C}$-modules is denoted $\rMod\calR_{\C}$. Given sheaves of $\calR_{\C}$-modules $\frakF$ and $\frakG$ we denote $\Hom_{\calR_{\C}}(\frakF, \frakG)$ the set of morphism of sheaves of $\calR_{\C}$-modules.
\end{definition}

Let $\C$ be a site and $R$ a commutative ring with identity. If $\calR_{\C}$ is the sheafification of the constant presheaf $\kappa_R$, then $\Sh(\C,R)=\rMod{\calR_{\C}}$, see for instance \cite[18.1]{KS}. We refer the reader to \cite[18.23.1]{St} for the definition of a coherent sheaf. The coherent objects of $\rMod{\calR_{\C}}$ form a subcategory $\coh(\calR_{\C})$. Now we are at the position to state results on $\rMod{\calR_{\pG}}$, $\rMod{\calR_{\C}}$ and $\rMod{\calR_G}$, as well as on the coherent objects. 

\begin{theorem} Suppose $\calP$ is a $G$-poset with an initial object and $\PG\to\calC$ is a category extension. We have an equivalence of ringed topoi
$$
(\Upsilon,\Upsilon^{\sharp}) : (\Sh(\GG),\calR_{\GG}) {\buildrel{\simeq}\over{\to}} (\Sh(\C),\calR_{\C}),
$$
where $\Upsilon=\Theta\Xi : \Sh(\GG)\to\Sh(\C)$ and $\Upsilon^{\sharp}:\Upsilon^{-1}(\calR_{\C})\to\calR_{\GG}$ is the identity. 
\begin{enumerate}
\item It induces an equivalence between module categories
$$
\rMod{\calR_{\C}} \simeq \rMod{\calR_{\GG}} \cong \rMod{RG}.
$$

\item If $R$ is noetherian, there is an equivalence for the categories of coherent sheaves
$$
\coh(\calR_{\C}) \simeq \coh(\calR_{\GG}) \cong \rmod{RG}.
$$
\end{enumerate}
\end{theorem}

\begin{proof} Following \cite[18.13]{St}, $\Upsilon$ in Theorem 3.8 leads to an equivalence of ringed topoi $(\Upsilon,\Upsilon^{\sharp}) : (\Sh(\GG),\calR_{\GG}) {\buildrel{\simeq}\over{\to}} (\Sh(\C),\calR_{\C})$. It induces two functors 
$$
\Upsilon_* : \rMod{\calR_{\GG}} \to \rMod{\calR_{\C}}
$$
and
$$
\Upsilon^*=\Upsilon^{-1} : \rMod{\calR_{\C}} \to \rMod{\calR_{\GG}},
$$
which still form an equivalence. 

The second statement follow because being coherent is an intrinsic property (in the sense that it is preserved by equivalences of ringed topoi), see \cite[18.18]{St} (and \cite[18.23]{St} for the definition of a coherent module). When $R$ is noetherian, $RG$ is noetherian. The coherent objects in $\rMod{RG}$ are just the finitely generated $RG$-modules.
\end{proof}

Because of the equivalences $\rMod{\calR_{\C}}\simeq\rMod{RG}$ and $\coh(\calR_{\C})\simeq\rmod{RG}$, many classical results on group representations can be readily translated to its counterpart of modules on $\C$.

When $\C=\OG$, the second statement is comparable to the one in the theory of Mackey functors, which says that the category of group representations is embedded into the category of Mackey functors ${\rm Mack}_R(G)$ \cite[Section 6]{TW}. In fact, the latter embedding is also $M\mapsto\frakF_M$, regarded as the fixed-point Mackey functor on $G$.

\end{document}